\documentclass[10pt]{article}

\usepackage{amssymb, amsmath, mathrsfs, amsopn, amsfonts, amstext}
\usepackage{graphicx, color, float, wrapfig, eqnarray,amsthm}
\usepackage[top=.8
in, bottom=1in, left=1in, right=1in]{geometry}

\usepackage{color}

\usepackage{tikz, tikz-3dplot,tkz-graph}
\usetikzlibrary{decorations.markings}
\tikzstyle{vertex}=[circle, draw, inner sep=0pt, minimum size=13pt]

\newcommand{\mf}{\mathfrak}
\newcommand{\dd}{\hspace{.1cm}|\hspace{.1cm}}
\newcommand{\A}{{\rm A}}

\newcommand{\diag}{\rm diag}
\newcommand{\ul}{\underline}
\newcommand{\ind}{{\rm ind \hspace{.1cm}}}
\newcommand{\bd}{\begin{description}}
\newcommand{\ed}{\end{description}}
\newcommand{\be}{\begin{enumerate}}
\newcommand{\ee}{\end{enumerate}}
\newcommand{\bi}{\begin{itemize}}
\newcommand{\ei}{\end{itemize}}
\newcommand{\bt}{\begin{tabular}}
\newcommand{\et}{\end{tabular}}
\newcommand{\beq}{\begin{equation}}
\newcommand{\eeq}{\end{equation}}
\newcommand{\beqs}{\begin{eqnarray*}}
\newcommand{\eeqs}{\end{eqnarray*}}

\newcommand{\ad}{\text{ad}}
\newcommand{\rk}{{\rm rk \hspace{.1cm}}}

\newcommand{\sright}[1]{\sigma(\overrightarrow{#1})}
\newcommand{\sleft}[1]{\sigma(\overleftarrow{#1})}
\newcommand{\tright}[1]{\tau(\overrightarrow{#1})}
\newcommand{\tleft}[1]{\tau(\overleftarrow{#1})}
\newcommand{\floor}[1]{\left \lfloor #1 \right \rfloor }

\newcommand{\sspace}{.16}

\newtheorem{theorem}{Theorem} 
\newtheorem{lemma}[theorem]{Lemma} 
 
\newtheorem{case}{Case}

\newtheorem{conjecture}[theorem]{Conjecture}

\theoremstyle{definition}
\newtheorem{definition}[theorem]{Definition}
\newtheorem{example}[theorem]{Example}
\newtheorem{remark}[theorem]{Remark}

\begin{document}
\title{\bf The unbroken spectrum of type-A Frobenius seaweeds}
\author{Vincent E. Coll, Jr., Matthew Hyatt, and
Colton Magnant}

\maketitle

\noindent
\textit{Department of Mathematics, Lehigh University, Bethlehem, PA, USA}\\
\textit{Mathematics Department, Pace University, Pleasantville, NY, USA}\\
\textit{Department of Mathematical Sciences, Georgia Southern University, Statesboro, GA, USA
}

\begin{abstract}
\noindent
If $\mathfrak{g}$ is a Frobenius Lie algebra, then for certain $F\in \mathfrak{g}^*$ the natural map $\mathfrak{g}\longrightarrow \mathfrak{g}^* $ 
given by $x \longmapsto F[x,-]$ is an isomorphism.  The inverse image of $F$ under this isomorphism is called a principal element.  We show that if $\mathfrak{g}$ is a Frobenius seaweed subalgebra of $A_{n-1}=\mathfrak{sl}(n)$ then the spectrum of the adjoint of a principal element consists of an unbroken set of integers whose multiplicites have a symmetric distribution.   Our proof methods are constructive and combinatorial in nature.  
\end{abstract}



\section{Introduction}
A \textit{biparabolic} subalgebra of a reductive Lie algebra 
$\mathfrak{g}$ is the intersection of two parabolic subalgebras whose sum is $\mathfrak{g}$.  They  were first introduced 
in the case $\mathfrak{g}=\mathfrak{sl}(n)$ by Dergachev and A. Kirillov \textbf{\cite{DK}} where such algebras were called Lie algebras of \textit{seaweed }type.  Associated to each seaweed subalgebra of the simple Lie algebra $A_{n-1}=\mathfrak{sl}(n)$ is a certain planar graph called a \textit{meander}. One of the main results of \textbf{\cite{DK}} is that the algebra's index may be computed by counting the connected components and cycles of its associated meander.\footnote{Recently, the current authors have extended this line of inquiry to seaweed subalgebras of the simple Lie algebra $C_n=\mathfrak{sp}(2n)$ by introducing the notion of a \textit{symplectic meander} \textbf{\textbf{\cite{CHM}}}  (see also \textbf{\cite{PY}}).}
Of particular interest are those seaweed algebras of $\mathfrak{sl}(n)$ whose associated meander graph consists a single path. Such algebras have index zero. More generally, algebras with index zero are called \emph{Frobenius}, and have been extensively studied in the context of invariant theory \textbf{\cite{Ooms}} and are of special interest in deformation theory and quantum group theory stemming from their connection to the classical Yang-Baxter equation.  See  \textbf{\cite{GG:Boundary}} and \textbf{\cite{GG:Frob}} for examples.

To fix the notation, let $\mathfrak{g}$ be a Lie algebra over $\mathbb{C}$.  For any functional $F \in \mathfrak{g}^*$ one can associate the skew-symmetric and bilinear \emph{Kirillov form} $B_F (x,y) = F[x,y]$.  The index of $\mathfrak{g}$ is defined to be the minimum dimension of the kernel of $B_{F}$ as $F$ ranges over $\mathfrak{g}^*$.  The Lie algebra $\mathfrak{g}$ is Frobenius if its index is zero.  Equivalently, $\mf{g}$ is Frobenius if there is a functional $F\in \mathfrak{g}^*$ such that $B_F$ is non-degenerate.  We call such an $F$ a \emph{Frobenius functional} and the natural map $\mathfrak{g} \rightarrow \mathfrak{g}^*$ defined by $x \mapsto  F[x,-]$ is an isomorphism.  The image of $F$ under the inverse of this map is called a \emph{principal element} of $\mathfrak{g}$ and will be denoted $\hat{F}$.  It is the unique element of $\mathfrak{g}$ such that 
$$
F\circ \mathrm{ad} \hat{F}= F([\hat{F},-]) = F.  
$$
If $\mathfrak{g}$ is Frobenius, its set of Frobenius functionals is, in general, quite large; forming an open subset of $\mathfrak{g}^{*}$ in the Zariski and Euclidean topologies (see \textbf{\cite{Ooms2}} and 
\textbf{\cite{DK}}).

As a consequence of Proposition 3.1 in \textbf{\cite{Ooms2}}, Ooms established that the spectrum of the adjoint of a principal element of a Frobenius Lie algebra is independent of the principal element chosen to compute it  (see also \textbf{\cite{G:Principal}}, Theorem 3).   Subsequently, Gerstenhaber and Giaquinto asserted, in particular (\textbf{\cite{G:Principal}}, Theorem 2), that if the algebra is a seaweed subalgebra of $\mathfrak{sl}(n)$ then the spectrum consists of an unbroken set of integers. However M. Dufflo, in a private communication to the authors, indicated that the proof of the unbroken property is incorrect. The primary goal of this article is to provide a complete and correct proof of the unbroken spectrum assertion, but we also show that the multiplicities of the spectrum's eigenvalues form a symmetric distribution (see Theorem \ref{GGConj}). The symmetry result is not an afterthought.  It is, in fact, critical to our proof of the unbroken spectrum result.  Our methods here are constructive and combinatorial in nature. 

\section{Seaweeds}

Let $\mf{p}$ and $\mf{p'}$ be two parabolic subalgebras of a reductive Lie algebra $\mf{g}$.  If $\mf{p} + \mf{p'}= \mf{g}$ then
$\mf{p}\cap\mf{p'}$ is called a \textit{seaweed}, or in the terminology of A. Joseph \textbf{\cite{Joseph}} \textit{biparabolic}, subalgebra of $\mf{g}$.  In what follows, we further assume that $\mf{g}$ is simple and comes equipped with a triangular decomposition 
\begin{eqnarray*} 
\mf{g}=\mf{u_+}\oplus\mf{h}\oplus\mf{u_-}
\end{eqnarray*}

\noindent
where $\mf{h}$ is a Cartan subalgebra of $\mf{g}$ and $\mf{u_+}$ and $\mf{u_-}$ are the subalgebras consisting of the upper and lower triangular matrices, respectively. Let $\Pi$ be the set of $\mf{g}$'s simple roots and for $\beta\in\Pi$,
let $\mf{g}_{\beta}$ denote the root space corresponding to $\beta$. A seaweed subalgebra $\mf{p}\cap\mf{p'}$ is called \textit{standard} if
$\mf{p}\supseteq \mf{h}\oplus\mf{u}_+$ and $\mf{p'}\supseteq \mf{h}\oplus\mf{u_-}$.
In the case that $\mf{p}\cap\mf{p'}$ is standard, let
$\Psi=\{\beta\in\Pi :\mf{g}_{-\beta}\notin \mf{p}\}$, 
$\Psi'=\{\beta\in\Pi :\mf{g}_{\beta}\notin \mf{p'}\}$,
and denote the seaweed by $\mf{p}(\Psi \dd \Psi')$.  

Let $\mf{sl}(n)$ be the algebra of $n\times n$ matrices with trace zero and consider the triangular decomposition of $\mf{sl}(n)$ as above.  Let $\Pi=\{\beta_1,\dots ,\beta_{n-1}\}$ be the set of simple roots of $\mf{sl}(n)$ with the standard ordering and let let $\mf{p}_n^\A(\Psi \dd \Psi')$ denote a seaweed subalgebra of $\mf{sl}(n)$ where $\Psi$ and $\Psi'$ are subsets of $\Pi$.  Let $C_n$ denote the set of sequences of positive integers whose sum is $n$ 
(i.e., $C_n$ is the set of compositions of $n$).
It will be convenient to index seaweeds of $\mf{sl}(n)$ by pairs of elements 
of $C_n$. Let $\mathcal{P}(X)$ denote the power set of a set $X$.
Let $\varphi_\A$ be the usual bijection from $C_n$ to a set of cardinality $n-1$.
That is, given $\ul{a}=(a_1,a_2,\dots ,a_m)\in C_n$, define
$\varphi_\A :C_n\rightarrow \mathcal{P}(\Pi)$ by
\[\varphi_\A(\ul{a})=\{\beta_{a_1},\beta_{a_1+a_2},\dots 
,\beta_{a_1+a_2+\dots +a_{m-1}}\}.\]

\noindent
Now, following the notational conventions established in \textbf{\cite{CGM11}} define the \textit{type} of the seaweed $\mf{p}_n^\A(\varphi_\A(\ul{a}) \dd \varphi_\A(\ul{b}))$ to be                                                                                                              the symbol 
\begin{eqnarray*}
\frac{a_{1}|a_{2}| \dots |a_{m}}{b_{1}|b_{2}| \dots |b_{t}},~~~ {\rm{where}}~~~\sum_{i=1}^m a_i = \sum_{j=1}^tb_j =n.
\end{eqnarray*}

By construction, the sequence of numbers in $\ul{a}$ determines the heights of triangles
below the main diagonal in $\mf{p}_n^\A(\varphi_\A(\ul{a}) \dd \varphi_\A(\ul{b}))$, which may have nonzero entries,
and the sequence of numbers in $\ul{b}$ determines the heights of triangles
above the main diagonal. 
For example, the seaweed 
$\mf{p}_6^\A(\{\beta_2\} \dd \{\beta_1,\beta_3\})$
of type $\frac{2|4}{1|2|3}$ has the following shape, where * indicates the possible nonzero entries. See 
Figure \ref{Aseaweed}.
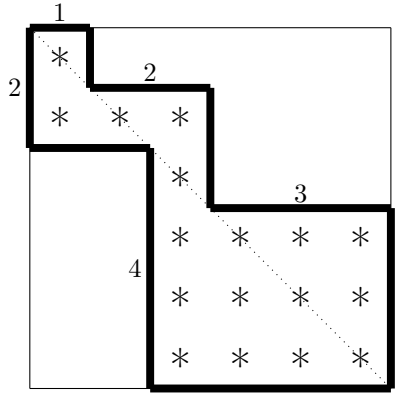
\begin{figure}[H]
\[\begin{tikzpicture}[scale=.8]
\draw (0,0) -- (0,6);
\draw (0,6) -- (6,6);
\draw (6,6) -- (6,0);
\draw (6,0) -- (0,0);
\draw [line width=3](0,6) -- (0,4);
\draw [line width=3](0,4) -- (2,4);
\draw [line width=3](2,4) -- (2,0);
\draw [line width=3](2,0) -- (6,0);

\draw [line width=3](0,6) -- (1,6);
\draw [line width=3](1,6) -- (1,5);
\draw [line width=3](1,5) -- (3,5);
\draw [line width=3](3,5) -- (3,3);
\draw [line width=3](3,3) -- (6,3);
\draw [line width=3](6,3) -- (6,0);

\draw [dotted] (0,6) -- (6,0);

\node at (.5,5.4) {{\LARGE *}};
\node at (.5,4.4) {{\LARGE *}};
\node at (1.5,4.4) {{\LARGE *}};
\node at (2.5,4.4) {{\LARGE *}};
\node at (2.5,3.4) {{\LARGE *}};
\node at (2.5,2.4) {{\LARGE *}};
\node at (3.5,2.4) {{\LARGE *}};
\node at (4.5,2.4) {{\LARGE *}};
\node at (5.5,2.4) {{\LARGE *}};
\node at (2.5,1.4) {{\LARGE *}};
\node at (3.5,1.4) {{\LARGE *}};
\node at (4.5,1.4) {{\LARGE *}};
\node at (5.5,1.4) {{\LARGE *}};
\node at (2.5,0.4) {{\LARGE *}};
\node at (3.5,0.4) {{\LARGE *}};
\node at (4.5,0.4) {{\LARGE *}};
\node at (5.5,0.4) {{\LARGE *}};

\node at (-.25,5) {2};
\node at (1.75,2) {4};
\node at (.5,6.25) {1};
\node at (2,5.25) {2};
\node at (4.5,3.25) {3};

\end{tikzpicture}
\]
\caption{
A seaweed of type $\frac{2|4}{1|2|3}$}
\label{Aseaweed}
\end{figure}

As noted in the introduction, a distinguished class of Lie algebras $\mathfrak{g}$ are those whose index is zero.  The \textit{index} of a Lie algebra $\mf{g}$ is an important invariant and may be regarded as a generalization of the Lie algebra's rank:  $\ind \mf{g}\le \rk  \mf{g}$, with equality when $\mf{g}$ is reductive.  More formally, the index of a Lie algebra $\mf{g}$ is given by

\[\ind \mf{g}=\min_{F\in \mf{g^*}} \dim  (\ker (B_F)),\]

\noindent where $B_F$ is the associated skew-symmetric \textit{Kirillov form} defined by $B_F(x,y)=F([x,y])$ for all $x,y\in\mf{g}$.  Such algebras are called \textit{Frobenius}.  We will be exclusively interested in Frobenius seaweed subalgebras of $\mathfrak{sl}(n)$ and from here on, we tacitly assume that all seaweeds are subalgebras of $\mathfrak{sl}(n)$.

\section{Meanders}

Following Dergechev and A. Kirillov \cite{DK}, we associate a planar graph 
to each seaweed of type 
$\frac{a_{1}|a_{2}| \dots |a_{m}}{b_{1}|b_{2}| \dots |b_{t}}$ 
as follows.
Line up $n$ vertices horizontally and label them $v_{1}, v_{2}, \dots, v_{n}$. Partition the set of vertices into two set partitions, called \emph{top} and \emph{bottom}. The top partition groups together the first $a_1$ vertices, then the next $a_2$ vertices, and so on, lastly grouping together the last $a_m$ vertices. In a similar way, the bottom partition is determined by the sequence $b_1,\dots ,b_t$. We call each set within a set partition a \textit{block}. For each block in the top (likewise bottom) partition we build up the graph by adding edges in the same way. First, add an edge from the first vertex of a block to the last vertex of the same block drawn concave down (respectively concave up in the bottom part case).  The edge addition is then repeated between the second vertex and the second to last and so on within each block of both partitions. 
More explicitly, given vertices $v_j,v_k$ in a top block of 
size $a_i$, there is an edge between them if and only if 
$j+k=2(a_1+a_2+\dots+a_{i-1})+a_i+1$. If $v_j,v_k$ are in a bottom block of 
size $b_i$, there is an edge between them if and only if 
$j+k=2(b_1+b_2+\dots+b_{i-1})+b_i+1$.
The resulting undirected planar graph 
is called the \textit{meander} associated to the given seaweed. 
We say that the meander has the 
same type as its associated seaweed. See the the left panel of Figure 
\ref{fig:meanders}.

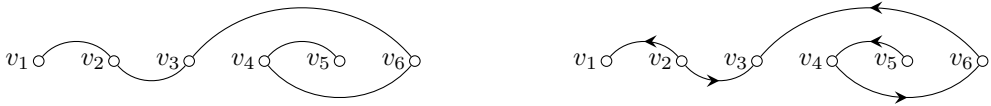
\begin{figure}[h]
\[\begin{tikzpicture}
\tikzstyle{every node}=[draw,circle,fill=white,minimum size=4pt,inner sep=0pt]
\tikzset{->-/.style={decoration={markings,
mark=at position #1 with {\arrow[scale=1.5,>=stealth]{>}}},postaction={decorate}}}
\draw (0,0) node[label=left:$v_1$] (0) {};
\draw (1,0) node[label=left:$v_2$] (1) {};
\draw (2,0) node[label=left:$v_3$] (2) {};
\draw (3,0) node[label=left:$v_4$] (3) {};
\draw (4,0) node[label=left:$v_5$] (4) {};
\draw (5,0) node[label=left:$v_6$] (5) {};


\draw (1) to [bend right=50] (0);
\draw (5) to [bend right=50] (2);
\draw (4) to [bend right=50] (3);
\draw (1) to [bend right=50] (2);
\draw (3) to [bend right=50] (5);

;\end{tikzpicture}
\hspace{2cm}
\begin{tikzpicture}
\tikzstyle{every node}=[draw,circle,fill=white,minimum size=4pt,inner sep=0pt]
\tikzset{->-/.style={decoration={markings,
mark=at position #1 with {\arrow[scale=1.5,>=stealth]{>}}},postaction={decorate}}}
\draw (0,0) node[label=left:$v_1$] (0) {};
\draw (1,0) node[label=left:$v_2$] (1) {};
\draw (2,0) node[label=left:$v_3$] (2) {};
\draw (3,0) node[label=left:$v_4$] (3) {};
\draw (4,0) node[label=left:$v_5$] (4) {};
\draw (5,0) node[label=left:$v_6$] (5) {};


\draw [->-=.5] (1) to [bend right=50] (0);
\draw [->-=.5] (5) to [bend right=50] (2);
\draw [->-=.5] (4) to [bend right=50] (3);
\draw [->-=.5] (1) to [bend right=50] (2);
\draw [->-=.5] (3) to [bend right=50] (5);

;\end{tikzpicture}
\]

\caption{A meander of type $\frac{2|4}{1|2|3}$}
\label{fig:meanders}
\end{figure}

One of the main results of \textbf{\cite{DK}} is that the index of a seaweed may be computed by counting the connected components and cycles of its associated meander. In particular, a seaweed is Frobenius precisely when its associated meander consists of single path.  In this case we also say that the meander is Frobenius.

\section{Frobenius functionals}

To each seaweed, there is a canonically defined functional developed by Dergachev and A. Kirillov in \textbf{\cite{DK}}.  
To present the definition of this functional, we require some preliminary notation.  Let $e_{ij}^*$ denote the linear functional which selects the $(i,j)$ entry of a given $n \times n$ matrix.  If $S$ is a subset of the indices $(i, j)$ with $1 \leq i, j \leq n$, then $F_{S}$ denotes the functional $\sum_{s \in S} e^{*}_{s}$.  We say that $S$ \emph{supports} $F_{S}$ and the \emph{directed graph of the functional}, $\Gamma(S)$, has vertices $v_{1}, \dots , v_{n}$ with an arrow from $v_{i}$ to $v_{j}$ whenever $(i, j) \in S$.  

If we begin with a meander graph and orient the top edges with a counterclockwise orientation and the bottom edges with clockwise orientation, the resulting directed meander graph is the directed graph of the Dergachev-Kirillov functional.
Henceforth, we shall assume that all meanders come equipped with this orientation,
unless otherwise stated. We let $MS$ denote the set of indices $(i,j)$ such that
there is a directed edge from $i$ to $j$ in $M$.
(See the right side of Figure 2 where the directed edges of $M$ are $\{(2,1),(6,3),(5,4),(2,3),(4,6)\}$ and $F_{MS}=e_{2,1}^*+e_{6,3}^*+e_{5,4}^*+e_{2,3}^*+e_{4,6}^*$.) The Dergachev-Kirillov functional is\textit{ regular} in the sense that $F_{MS}$ realizes the smallest possible dimension for the kernel of 
$F([y,-])$ where $F$ ranges over all linear functionals.  In particular, $F_{MS}$ is Frobenius precisely when $M$ is Frobenius.

\section{Principal Elements}\label{sec:principal}

If $\mathfrak{g}$ is a Frobenius Lie algebra, then for a Frobenius functional $F\in \mathfrak{g}^*$ the natural map $\mathfrak{g}\longrightarrow \mathfrak{g}^* $ given by $x \longmapsto F[x,-]$ is an isomorphism.  The inverse image of $F$ under this isomorphism is called a \textit{principal element} of $\mathfrak{g}$ (corresponding to $F$) and will be denoted $\hat{F}$.   It is the unique element of $\mathfrak{g}$ such that $F([\hat{F},-]) = F$. As noted in the introduction, the spectrum of the adjoint representation of $\hat{F}$ is independent of which principal element is used to compute it.



In practice, given a Frobenius seaweed $\mf{g}\subset \mf{sl}(n)$, it is desirable to construct $\hat{F}$ and the spectrum of ad $\hat{F}$ directly from its directed meander graph $M$. Following Gerstenhaber and Giaquinto  we have the following construction (see \textbf{\cite{GG:Frob}} and {\textbf{\cite{G:Principal}}}).
Define the \textit{measure} of path from vertex $v_i$ to vertex $v_j$
in $M$ to be the number of forward edges minus the number of backward edges encountered
when moving along the path from $v_i$ to $v_j$.
Fix any vertex $v$ in $M$. Construct a diagonal matrix $D$
by setting the $i^{\text{th}}$ diagonal entry of $D$ equal to
the measure of the path from $v_i$ to $v$. 
Then $\hat{F}$ is obtained from $D$ by adding an appropriate multiple of the identity matrix, so that the resulting matrix has trace zero.  For all $e_{i,j}\in\mf{g}$ such that $i\neq j$, $e_{i,j}$
is an eigenvector of $\ad \hat{F}:\mf{g}\rightarrow \mf{g}$.
The corresponding eigenvalue
is the measure of the path from $v_i$ to $v_j$ in $M$. Additionally,
$\ad \hat{F}(h)=0$ for all $h\in\mf{h}$, so $n-1$ zeros must be added
to the list of eigenvalues.

\begin{example}\label{ex:principal}  Consider the meander $M$ of type $\frac{2|4}{1|2|3}$ shown in Figure \ref{fig:meander}.  Since $M$ is Frobenius, $F_{MS}=e_{2,1}^*+e_{6,3}^*+e_{5,4}^*+e_{2,3}^*+e_{4,6}^*$ is Frobenius.

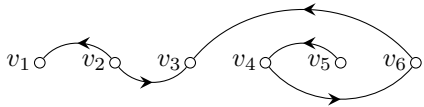
\begin{figure}[h]
\[\begin{tikzpicture}
\tikzstyle{every node}=[draw,circle,fill=white,minimum size=4pt,inner sep=0pt]
\tikzset{->-/.style={decoration={markings,
mark=at position #1 with {\arrow[scale=1.5,>=stealth]{>}}},postaction={decorate}}}
\draw (0,0) node[label=left:$v_1$] (0) {};
\draw (1,0) node[label=left:$v_2$] (1) {};
\draw (2,0) node[label=left:$v_3$] (2) {};
\draw (3,0) node[label=left:$v_4$] (3) {};
\draw (4,0) node[label=left:$v_5$] (4) {};
\draw (5,0) node[label=left:$v_6$] (5) {};


\draw [->-=.5] (1) to [bend right=50] (0);
\draw [->-=.5] (5) to [bend right=50] (2);
\draw [->-=.5] (4) to [bend right=50] (3);
\draw [->-=.5] (1) to [bend right=50] (2);
\draw [->-=.5] (3) to [bend right=50] (5);

;\end{tikzpicture}\]

\caption{The directed graph of the meander of type $\frac{2|4}{1|2|3}$}
\label{fig:meander}
\end{figure}

The principal element is computed as follows. Using $v=v_5$, we have $D=\diag(-3,-2,-3,-1,0,-2),$
and so
\[\hat{F}=D+\frac{11}{6}=\diag\left(-\frac{7}{6},-\frac{1}{6},-\frac{7}{6},
\frac{5}{6},\frac{11}{6},-\frac{1}{6}\right).\]

Now, to represent the eigenvalues and their corresponding eigenvectors, we make use of a $6 \times 6$ bookkeeping matrix wherein each bold entry is the eigenvalue of $e_{i,j}$, provided $i\neq j$. The bold zero in the $(i,i)^{\text{th}}$ entry represents the eigenvalue of $e_{i,i}-e_{i+1,i+1}$. See Figure \ref{fig:eigenvalues}.

\begin{figure}[H]
$$
\left[ \begin{array}{cccccc}
{\bf 0} & 0 & 0 & 0 & 0 & 0\\
{\bf 1} & {\bf 0} & {\bf 1} & 0 & 0 & 0\\
0 & 0 & {\bf 0} & 0 & 0 & 0\\
0 & 0 & {\bf 2} & {\bf 0} & {\bf -1} & {\bf 1}\\
0 & 0 & {\bf 3} & {\bf 1} & {\bf 0} & {\bf 2}\\
0 & 0 & {\bf 1} & {\bf -1} & {\bf -2} & 0
\end{array} \right].
$$
\caption{The eigenvalues of the seaweed of type $\frac{2|4}{1|2|3}$}
\label{fig:eigenvalues}
\end{figure}

\noindent
Thus, the eigenvalues are given by the multiset $\{-2,-1,-1,0,0,0,0,0,1,1,1,1,1,2,2,3\}$.  Note also for this example that if we let $d_i$ be the dimension of the eigenvalue $i$, then the sequence $\{d_i\}$ is both symmetric and unimodal about one half.
\end{example} 

\begin{remark}  It is not surprising that the eigenvalues in the above example are integers. If $\mf{g}$ is a Frobenius subalgebra of $\mf{sl}(n)$, and $F_S$ is a Frobenius functional such that as an undirected graph $\Gamma(S)$ is a tree, then the Gerstenhaber-Giaquinto construction of a principal element is valid. Therefore, whenever $\Gamma(S)$ is a tree, the spectrum consists only of integers.  
\end{remark}

Since the spectrum of the adjoint of the principal element
is an invariant of a Frobenius Lie algebra, we may unambiguously refer to it
as the spectrum of the algebra.   We are now in a position to state the main result of this paper which is the following theorem.

\begin{theorem}\label{GGConj}
For $n\geq 2$, the spectrum of a Frobenius seaweed subalgebra of $\mathfrak{sl}(n)$ consists of an unbroken set of integers centered at one half.  Moreover, the dimensions of the associated eigenspaces form a symmetric distribution about one half. 
\end{theorem}

\section{The proof of Theorem \ref{GGConj}}\label{sec:proof}
We address the two items in the theorem separately, dealing symmetry in Section 6.1 and the unbroken property in Section 6.2. As noted in the introduction, the symmetry result is essential to establish the unbroken property.

\subsection{Symmetry}

Given a Frobenius meander $M$ of type $\frac{a_{1}|a_{2}| \dots |a_{m}}{b_{1}|b_{2}| \dots |b_{t}}$, let $m(i,j)$ denote the measure of the 
path from $v_i$ to $v_j$. Given a top block of vertices $A$ of size $a_k$,
we say that the eigenvalues contributed by $A$ is the following multiset
\begin{equation}\label{top contribute}
\{m(i,j)\mid v_i,v_j\in A\text{ and }i>j\}\cup\{0^{\floor{\frac{a_k}{2}}}\},
\end{equation}
where $\{ q^{l} \}$ denotes a multiset consisting of $q$ repeated $l$ times.
Similarly, given a bottom block of vertices $B$ of size $b_k$,
we say that the eigenvalues contributed by $B$ is the following multiset
\begin{equation}\label{bottom contribute}
\{m(i,j)\mid v_i,v_j\in B\text{ and }i<j\}\cup\{0^{\floor{\frac{b_k}{2}}}\}.
\end{equation}
Note that $\sum \floor{\frac{a_k}{2}}+\floor{\frac{b_k}{2}}=n-1$,
since a block of size $a_k$ has precisely $\floor{\frac{a_k}{2}}$ edges,
and there are $n-1$ edges in $M$. It follows that
the eigenvalues contributed by each block is a multiset partition of all the
eigenvalues of the seaweed. The following lemma implies the symmetry
property of Theorem \ref{GGConj}.


\begin{lemma}\label{thm:symmetric}
Let $A$ be any top or bottom block of a Frobenius seaweed of $\mf{sl}(n)$
with $n\geq 2$.
Given any integer $l$, let $d_l$ be the multiplicity of the 
eigenvalue $l$ in the multiset of eigenvalues contributed by $A$.
The sequence $(d_l)$ is symmetric about one half.
\end{lemma}

\begin{proof}  Let $A=\{v_s,v_{s+1},\dots ,v_{s+a_k-1}\}$ be a top block of size $a_k$. The argument when $A$ is a bottom
block is similar and will be omitted. 
We will partition the eigenvalues contributed by $A$ into pairs, such that
the mean of each pair is one half.

For each pair of vertices $v_i,v_j\in A$ such that $i>j$ and 
$i+j=a_k+2s-1$, pair the eigenvalue $m(i,j)$ with
one of the $\floor{\frac{a_k}{2}}$ zeros from \eqref{top contribute}
that is not a measure of a path. By the definition of the meander $M$
we have $m(i,j)=1$, thus the mean of each pair is one half. 

For each pair of vertices $v_i,v_j\in A$ such that $i>j$ and 
$i+j>a_k+2s-1$, consider the
pair of eigenvalues $\{m(i,j),m(j,a_k+2s-1-i)\}$. Note that 
$j>a_k+2s-1-i\geq s$, thus $m(j,a_k+2s-1-i)$ is indeed an eigenvalue
contributed by $A$. Also note that
\[j+(a_k+2s-1-i)=a_k+2s-1+(j-i)\leq a_k+2s-2.\]
This means that each eigenvalue $m(i,j)$ such that $i+j>a_k+2s-1$ 
can be paried with a unique eigenvalue $m(j,a_k+2s-1-i)$ such that 
$j+(a_k+2s-1-i)\leq a_k+2s-2$. Lastly, note that
\[m(i,j)+m(j,a_k+2s-1-i)=m(i,a_k+2s-1-i)=1,\]
by the definition of $M$. Therefore, the mean
of each pair is again one half.
\end{proof}

\subsection{Unbroken}

Every Frobenius meander can be constructed by a unique sequence of graph theoretic moves, which we call Winding-up moves (see Lemma \ref{thm:Contraction}).  These moves are 
the inverse of Panyushev's 
formula for the index of a seaweed \textbf{\cite{Panyushev}}. We will prove the unbroken property by induction on the number of Winding-up moves.

Our inductive hypothesis will be that the set of eigenvalues within each block
(see $\sleft{A}$ and $\sright{A}$ from equations \eqref{eq:sleft} and \eqref{eq:sright} below)
of a Frboenius meander is in fact an unbroken set. We will also need an 
additional inductive hypothesis that a certain subset of the eigenvalues
within each block is an unbroken set (see $\tleft{A}$
and $\tright{A}$ from equations \eqref{eq:tleft} and \eqref{eq:tright}
below). In the inductive step of the proof, we rely on the properties
of what we call simple eigenvalues, which we define next.

\begin{definition}
If $v_i$ and $v_{i+1}$ are the same top (respectively bottom) block of
a Frobenius meander,
we call $m(i+1,i)$ (respectively $m(i,i+1)$) a \textit{simple eigenvalue}.
\end{definition}

The simple eigenvalues are so named since they are the eigenvalues
(of the adjoint action of the principal element) of
the simple root vectors of $\mf{sl}(n)$ contained in the seaweed.
The following lemma details the properties of simple eigenvalues that will be used in the proof of the unbroken spectrum property.

\begin{lemma}\label{prop:simple}
Let $M$ be a Frobenius meander.  We have the following items.

\be

\item\label{prop1} Every eigenvalue $m(k,j)$ of $M$ is a sum 
of simple eigenvalues. Namely $m(k,j)=\sum_{i=k}^{j-1}m(i,i+1)$ if $k<j$,
or $m(k,j)=\sum_{i=j}^{k-1}m(i+1,i)$ if $j<k$.

\item\label{prop2} If there is an edge from $v_{i+1}$ to $v_{i}$, then $m(i+1,i)=1$.
If there is an edge from $v_i$ to $v_{i+1}$, then $m(i,i+1)=1$.
This can only happen at the center of a block of even size.

\item\label{prop3} If there is an edge from $v_{i+2}$ to $v_{i}$, 
then $m(i+2,i+1)+m(i+1,i)=1$.
If there is an edge from $v_{i}$ to $v_{i+2}$, 
then $m(i,i+1)+m(i+1,i+2)=1$.
This can only happen at the center of a block of odd size.

\item\label{prop4} If $v_i,v_{i+1}$ are in the same top block, and in the same bottom
block, then $m(i+1,i)=-m(i,i+1)$.

\item\label{prop5} If there is an edge from $v_{i+j}$ to $v_{i}$ where
$j\geq 3$, then $m(i+j,i+j-1)=-m(i+1,i)$.
If there is an edge from $v_{i}$ to $v_{i+j}$ where
$j\geq 3$, then $m(i,i+1)=-m(i+j-1,i+j)$.

\ee
\end{lemma}

\begin{proof}
Items \ref{prop1}, \ref{prop2}, \ref{prop3}, and \ref{prop4} are obvious. 
For Item \ref{prop5} note that if there is an edge from
$v_{i+j}$ to $v_{i}$, then
\[1=\sum_{k=0}^{j-1}m(i+k+1,i+k).\]
Since $j\geq 3$ there is also an edge from $v_{i+j-1}$ to $v_{i+1}$,
it follows that
\[1=\sum_{k=1}^{j-2}m(i+k+1,i+k).\]
A similar argument holds when there is an edge from $v_{i}$ to $v_{i+j}$.
\end{proof}

All of the simple eigenvalues can be efficiently computed as follows.
Start from the center of a block of even size and apply Item \ref{prop2} 
(if $M$
is Frobenius and $n\geq 2$, then there is at least one block of even size).
Next apply Item \ref{prop4}, and then apply either 
Item \ref{prop3} or \ref{prop5}. Repeat the process of
applying Items \ref{prop4} followed by Item \ref{prop3} or \ref{prop5}, 
until Item \ref{prop4} cannot be applied.
Repeat this procedure for each block of even size.
Since $M$ is a Frobenius meander (i.e. a single path), every simple eigenvalue
can be determined from this process.

In Figure \ref{fig:simple} below, simple eigenvalues $m(i+1,i)$ are placed
between $v_i,v_{i+1}$ above the horizontal line of the vertices,
and simple eigenvalues $m(i,i+1)$ are placed below. (The dotted horizontal
line is a visual aid to distinguish between simple eigenvalues - it is not an edge.)
\begin{figure}[H]
\[\begin{tikzpicture}[scale=1]
\tikzstyle{every node}=[draw,circle,fill=white,minimum size=4pt,inner sep=0pt]
\tikzset{->-/.style={decoration={markings,
mark=at position #1 with {\arrow[scale=1.5,>=stealth]{>}}},postaction={decorate}}}
\node[draw=none] at (4.5,-\sspace) {1};
\node[draw=none] at (4.5,\sspace) {-1};
\node[draw=none] at (7.5,\sspace) {1};
\node[draw=none] at (7.5,-\sspace) {-1};
\node[draw=none] at (1.5,-\sspace) {1};
\node[draw=none] at (1.5,\sspace) {-1};
\node[draw=none] at (-.5,\sspace) {1};
\node[draw=none] at (-.5,-\sspace) {-1};
\node[draw=none] at (9.5,-\sspace) {1};
\node[draw=none] at (9.5,\sspace) {-1};
\node[draw=none] at (10.5,\sspace) {2};
\node[draw=none] at (10.5,-\sspace) {-2};
\node[draw=none] at (-1.5,-\sspace) {2};
\node[draw=none] at (-1.5,\sspace) {-2};
\node[draw=none] at (2.5,\sspace) {2};
\node[draw=none] at (2.5,-\sspace) {-2};
\node[draw=none] at (6.5,-\sspace) {2};
\node[draw=none] at (6.5,\sspace) {-2};
\node[draw=none] at (5.5,\sspace) {3};
\node[draw=none] at (5.5,-\sspace) {-3};
\node[draw=none] at (3.5,-\sspace) {3};
\node[draw=none] at (.5,\sspace) {1};
\node[draw=none] at (.5,-\sspace) {-1};
\node[draw=none] at (8.5,-\sspace) {1};

\draw (-2,0) node (-2) {};
\draw (-1,0) node (-1) {};
\draw (0,0) node (0) {};
\draw (1,0) node (1) {};
\draw (2,0) node (2) {};
\draw (3,0) node (3) {};
\draw (4,0) node (4) {};
\draw (5,0) node (5) {};
\draw (6,0) node (6) {};
\draw (7,0) node (7) {};
\draw (8,0) node (8) {};
\draw (9,0) node (9) {};
\draw (10,0) node (10) {};
\draw (11,0) node (11) {};

\draw[dotted] (-2.5,0) -- (11.5,0);
\draw [->-=.5] (3) to [bend right=50] (-2);
\draw [->-=.5] (2) to [bend right=50] (-1);
\draw [->-=.5] (1) to [bend right=80] (0);
\draw [->-=.5] (8) to [bend right=50] (4);
\draw [->-=.5] (7) to [bend right=50] (5);
\draw [->-=.5] (-2) to [bend right=50] (11);
\draw [->-=.5] (-1) to [bend right=50] (10);
\draw [->-=.5] (0) to [bend right=50] (9);
\draw [->-=.5] (1) to [bend right=50] (8);
\draw [->-=.5] (2) to [bend right=50] (7);
\draw [->-=.5] (3) to [bend right=50] (6);
\draw [->-=.5] (4) to [bend right=80] (5);
\draw [->-=.5] (11) to [bend right=50] (9);
;\end{tikzpicture}
\]

\caption{Simple eigenvalues for the meander $\dfrac{6|5|3}{14}$}
\label{fig:simple}
\end{figure}
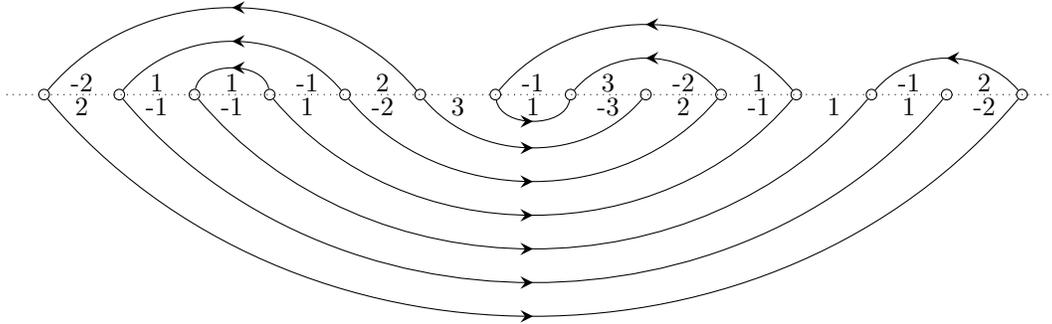

\begin{lemma}\label{lem:simple}
The absolute value of every simple eigenvalue is either 1, 2, or 3.
\end{lemma}

\begin{proof}
When using Lemma \ref{prop:simple} to determine the simple eigenvalues,
the first simple eigenvalue is positive. Item \ref{prop4} is then applied,
and the next simple eigenvalue is negative.
One then uses Item \ref{prop3} or \ref{prop5} to obtain the next 
simple eigenvalue, which must be 
positive. Since this process is repeated, the pair of simple eigenvalues
involved in using Item \ref{prop3} the first time is -1 and 2. 
If Item \ref{prop3}
is applied a second time before starting from the center 
of another block of even size,
then the pair of simple eigenvalues involved is -2 and 3.

Since the meander is Frobenius with $n\geq 2$, there are exactly two blocks
of odd size. This means Item \ref{prop3} can only be applied at most two times, 
and subsequent applications of Items \ref{prop4} and \ref{prop5} 
do not change the absolute value of the simple eigenvalues.
\end{proof}

The following Lemma\footnote{A superset of these Winding-down moves was first introduced in \textbf{\cite{CMW}} where the \textit{Signature of a meander} was introduced and used to develop index formulas for Type-A seaweeds (see also \textbf{\cite{CHMW})}.} is a special case of the well-known formula of Panyushev (see \textbf{\cite{Panyushev}}, Theorem 4.2).  We find that any Frobenius meander can be contracted to the empty meander through a deterministic sequence of graph-theoretic moves, each of which is uniquely determined by the structure of the meander at the time of move application.

\begin{lemma}[Winding-down]\label{thm:Contraction}
Let $M$ be a Frobenius meander of $\mf{sl}(n)$ with $n\geq 2$, 
and let $M$ have type
$
\frac{a_{1}|a_{2}| \dots |a_{m}}{b_{1}|b_{2}| \dots |b_{t}}.
$
We have the following moves:
\be
\item {\bf Block Elimination:} If $a_{1} = 2b_{1}$, then
$\displaystyle
M \mapsto \frac{b_{1}|a_{2}|a_{3}| \dots |a_{m}}{b_{2}|b_{3}| \dots |b_{t}}.
$
\item {\bf Rotation Contraction:} If $b_{1} < a_{1} < 2b_{1}$, then
$\displaystyle
M \mapsto \frac{b_{1}|a_{2}|a_{3}| \dots |a_{m}}{(2b_{1} - a_{1})|b_{2}|b_{3}| \dots |b_{t}}.
$
\item {\bf Pure Contraction:} If $a_1>2b_1$, then
$\displaystyle
M \mapsto \frac{a_{1} - 2b_{1}|b_{1}|a_{2}|a_{3}| \dots |a_{m}}{b_{2}|b_{3}| \dots |b_{t}}.
$
\item {\bf Flip-Down:} If $a_1<b_1$, then 
$\displaystyle
M \mapsto \frac{b_{1}|b_{2}|b_3|\dots|b_t}{{a_{1}|a_{2}|a_3|\dots|a_m}}.  
$

\ee
\end{lemma}

Note that we need not consider the case $a_1=b_1$ since we are presuming that the meander is Frobenius and so homotopically trivial.  The following result is immediate by simply reversing each move in 
Lemma~\ref{thm:Contraction}.

\begin{lemma}[Winding-up]\label{Expansion}
Every Frobenius meander with at 
least 2 vertices is the result of a 
sequence of the following moves applied to the one vertex meander.
Given a meander 
$M$ of type $\frac{a_{1}|a_{2}| \dots |a_{m}}{b_{1}|b_{2}| \dots |b_{t}}$, 
create a meander $M'$ by one of the following moves:

\be
\item {\bf Block Creation:}
$M\mapsto M'$ of type $\displaystyle
\frac{2a_{1}|a_{2}| \dots |a_{m}}{a_{1}|b_{1}|b_{2}| \dots |b_{t}}.$
\item {\bf Rotation Expansion:}
$M\mapsto M'$ of type $\displaystyle
\frac{(2a_{1} - b_{1})|a_{2}| \dots |a_{m}}{a_{1}|b_{2}| \dots |b_{t}},$
provided that $a_{1} > b_{1}$.
\item {\bf Pure Expansion:}
$M\mapsto M'$ of type $\displaystyle
\frac{a_{1} + 2a_{2}|a_{3}|a_{4}|\dots|a_{m}}{a_{2}|b_{1}|b_{2}| \dots |b_{t}}.$
\item {\bf Flip-Up:} $M\mapsto M'$ of type 
$\displaystyle\frac{b_{1}|b_{2}|b_3|\dots|b_t}{{a_{1}|a_{2}|a_3|\dots|a_m}}.$

\ee
\end{lemma}

Lastly, we define the sets that will be needed in
the statement and proof of Lemma \ref{thm:unbroken} below,
which establishes that the set eigenvalues from each block
is unbroken.
Let $M$ be a Frobenius meander and let $A$ be a subset of its vertices.
Define the following sets of measures (note
that multisets are not needed to prove the unbroken property).
\begin{equation}\label{eq:sleft}
\sigma(\overleftarrow{A})=\{m(i,j)\mid v_i,v_j\in A\text{ and }i\geq j\},
\end{equation}
\begin{equation}\label{eq:sright}
\sigma(\overrightarrow{A})=\{m(i,j)\mid v_i,v_j\in A\text{ and }i\leq j\}.
\end{equation}
Also, if $A$ has leftmost vertex $v_j$ and rightmost vertex
$v_p$, define
\begin{equation}\label{eq:tleft}
\tau(\overleftarrow{A})=\{m(i,j)\mid j\leq i\leq p\},
\end{equation}
\begin{equation}\label{eq:tright}
\tau(\overrightarrow{A})=\{m(i,p)\mid j\leq i\leq p\}.
\end{equation}

\begin{lemma}\label{thm:unbroken}
For $n\geq 2$, if $M$ is Frobenius meander with $n$ vertices, then
$\sigma(\overleftarrow{A})$ is an unbroken set for every top block $A$,
and $\sigma(\overrightarrow{B})$ is an unbroken set for every bottom block
$B$.

\end{lemma}

\begin{proof} Our proof is by  
induction on the number of Winding-up moves from Lemma \ref{Expansion},
that $\sigma(\overleftarrow{A})$ and $\tau(\overleftarrow{A})$ are unbroken
for every top block, and $\sigma(\overrightarrow{B})$ and 
$\tau(\overrightarrow{B})$ are unbroken for every bottom block.
This will prove Lemma \ref{thm:unbroken} since each unbroken set contains
zero. We note that we will need the inductive hypothesis regarding a set of
the form $\tau(\overleftarrow{A})$ to prove that a set of the form
$\sleft{A}$ is unbroken when a 
Pure Expansion moved is performed (see equation \eqref{pe5} below).  Also, we need not consider the Flip-up move since this merely replaces $M$ with an inverted isomorphic copy and so has no effect on our eigenvalue calculations.

The base of the induction is the Frobenius seaweed $\frac{2}{1|1}$.
For the top block we have 
$\sigma(\overleftarrow{A})=\tau(\overleftarrow{A})=\{0,1\}$.
For each bottom block 
$\sigma(\overrightarrow{B})=\tau(\overrightarrow{B})=\{0\}.$  Now, 
let $M$ be a Frobenius meander of type 
$\frac{a_1|a_2|\dots |a_m}{b_1|b_2|\dots |b_t}$. 
By way of induction, we assume that the all the sets mentioned 
in our inductive hypothesis are unbroken.
In what follows we break the proof into three cases based on 
which of the Winding-up moves is performed to produce a larger Frobenius meander 
$M'$.
We will make use of the following notation.
Given any sets $S, T$ and an integer $c$, define $S + c=\{s+c\mid s\in S\}$, 
$-S=\{-s\mid s\in S\}$, and $S+T=\{s+t\mid s\in S \text{ and }t\in T\}$.
Given set of vertices $A$ and $B$
where all the vertices of $A$ lie to the left of all the vertices in $B$,
define 
$\displaystyle\sigma(A\leftarrow B)=\{m(j,i)\mid v_i\in A\text{ and }v_j\in B\}$,
and
$\displaystyle\sigma(A\rightarrow B)=\{m(i,j)\mid v_i\in A\text{ and }v_j\in B\}$.

\setcounter{case}{0}
\begin{case}\label{bc}
A Block Creation move is performed.
\end{case}

Let $A$ be the set consisting of the first $a_{1}$ vertices of $M$, namely
$A=\{v_1,\dots ,v_{a_1}\}$.  After the Block Creation is performed, let $A'$ be the set consisting of the first $a_{1}$ vertices of $M'$, namely $A'=\{v_1,\dots ,v_{a_1}\}$. Let $B'$ be the set consisting of the next $a_{1}$ vertices of $M'$, namely $B'=\{v_{a_{1} + 1}, \dots, v_{2a_{1}}\}$. 

Let $k=a_1-1$ and for $i=1,\dots ,k$ let $\alpha_i=m(i+1,i)$.  In this notation, the simple eigenvalues of first top block of $M$, read from left to right are $\alpha_1,\dots ,\alpha_{k}$.  Similarly, the simple eigenvalues of first top block of $M'$, also read from left to right are $\alpha_k,\dots ,\alpha_1$. 
The simple eigenvalues for the first top block of $M'$ are 
$-\alpha_k,\dots ,-\alpha_1,1,\alpha_1,\dots ,\alpha_k$. And the simple eigenvalues for every other block of $M'$ are the same 
as the simple eigenvalues of the block of corresponding size in $M$. 
For example, see Figure \ref{fig:blockcreation} below where $k=3$.

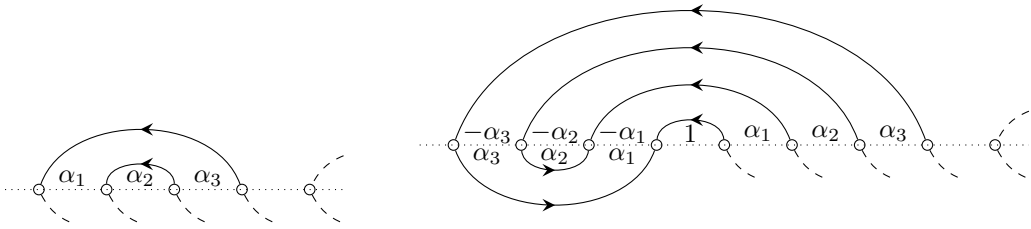
\begin{figure}[H]
\[\begin{tikzpicture}[scale=.9]
\tikzstyle{every node}=[draw,circle,fill=white,minimum size=4pt,inner sep=0pt]
\tikzset{->-/.style={decoration={markings,
mark=at position #1 with {\arrow[scale=1.5,>=stealth]{>}}},postaction={decorate}}}
\node[draw=none] at (1.5,.17) {$\alpha_1$};
\node[draw=none] at (2.5,.17) {$\alpha_2$};
\node[draw=none] at (3.5,.17) {$\alpha_3$};

\draw (1,0) node (1) {};
\draw (2,0) node (2) {};
\draw (3,0) node (3) {};
\draw (4,0) node (4) {};
\draw (5,0) node (5) {};

\draw[dotted] (0.5,0) -- (5.5,0);
\draw [->-=.5] (4) to [bend right=70] (1);
\draw [->-=.5] (3) to [bend right=90] (2);
\draw [dashed] (5.5,.5) to [bend right=20] (5);
\draw [dashed] (1) to [bend right=20] (1.5,-.5);
\draw [dashed] (2) to [bend right=20] (2.5,-.5);
\draw [dashed] (3) to [bend right=20] (3.5,-.5);
\draw [dashed] (4) to [bend right=20] (4.5,-.5);
\draw [dashed] (5) to [bend right=20] (5.5,-.5);

;\end{tikzpicture}
\hspace{1cm}
\begin{tikzpicture}[scale=.9]
\tikzstyle{every node}=[draw,circle,fill=white,minimum size=4pt,inner sep=0pt]
\tikzset{->-/.style={decoration={markings,
mark=at position #1 with {\arrow[scale=1.5,>=stealth]{>}}},postaction={decorate}}}
\node[draw=none] at (1.5,.17) {$-\alpha_3$};
\node[draw=none] at (2.5,.17) {$-\alpha_2$};
\node[draw=none] at (3.5,.17) {$-\alpha_1$};
\node[draw=none] at (4.5,.17) {1};
\node[draw=none] at (5.5,.17) {$\alpha_1$};
\node[draw=none] at (6.5,.17) {$\alpha_2$};
\node[draw=none] at (7.5,.17) {$\alpha_3$};
\node[draw=none] at (1.5,-.17) {$\alpha_3$};
\node[draw=none] at (2.5,-.17) {$\alpha_2$};
\node[draw=none] at (3.5,-.17) {$\alpha_1$};

\draw (1,0) node (1) {};
\draw (2,0) node (2) {};
\draw (3,0) node (3) {};
\draw (4,0) node (4) {};
\draw (5,0) node (5) {};
\draw (6,0) node (6) {};
\draw (7,0) node (7) {};
\draw (8,0) node (8) {};
\draw (9,0) node (9) {};

\draw[dotted] (0.5,0) -- (9.5,0);
\draw [->-=.5] (8) to [bend right=70] (1);
\draw [->-=.5] (7) to [bend right=70] (2);
\draw [->-=.5] (6) to [bend right=70] (3);
\draw [->-=.5] (5) to [bend right=85] (4);
\draw [->-=.5] (1) to [bend right=70] (4);
\draw [->-=.5] (2) to [bend right=85] (3);
\draw [dashed] (9.5,.5) to [bend right=20] (9);
\draw [dashed] (5) to [bend right=20] (5.5,-.5);
\draw [dashed] (6) to [bend right=20] (6.5,-.5);
\draw [dashed] (7) to [bend right=20] (7.5,-.5);
\draw [dashed] (8) to [bend right=20] (8.5,-.5);
\draw [dashed] (9) to [bend right=20] (9.5,-.5);

;\end{tikzpicture}
\]

\caption{Block Creation applied to $M$ (left) to obtain $M'$ (right), with
simple eigenvalues shown}
\label{fig:blockcreation}
\end{figure}

For the new bottom block of size $a_1$, clearly $\sright{A'}=\sleft{A}$
and $\tright{A'}=\tleft{A}$, and by induction both $\sleft{A}$ and
$\tleft{A}$ are unbroken.  For the new top block of size $2a_1$, we claim that the following equations hold:
\begin{eqnarray}
\sleft{A'}&=& -\sleft{A},\label{bc1}\\
\sleft{B'}&=&\sleft{A},\label{bc2}\\
\sigma(A' \leftarrow B') 
&=& \left[\sleft{A} + 1\right] 
\cup \left[-\sleft{A} + 1\right]\label{bc3},\\
\tleft{A'\cup B'}&=&\left[\tleft{A}-1\right]\cup\tleft{A}.\label{bc4}
\end{eqnarray}


Equations \eqref{bc1}, \eqref{bc2}, and \eqref{bc3} are obvious.
To prove equation \eqref{bc4}, let $q$ be an integer
satisfying $k+1\geq q\geq 1$.
Note that $\sum_{i=1}^{k}\alpha_i=1$, so we have 
\[-\alpha_k-\alpha_{k-1}-\dots -\alpha_{q}
=\alpha_1+\alpha_2+\dots +\alpha_{q-1}-1\in\tleft{A}-1.\]
If $r$ is an integer satisfying $0\leq r\leq k$, then
\[-\alpha_k-\alpha_{k-1}-\dots -\alpha_1+1+\alpha_1+\alpha_2+\dots +\alpha_{r}
=\alpha_1+\alpha_2+\dots +\alpha_{r}\in\tleft{A}.\]

Since $\sleft{A}$ is an unbroken set containing 0, it follows
that the union of the right hand sides of equations \eqref{bc1}, 
\eqref{bc2}, and \eqref{bc3} is unbroken.   And
since $\tleft{A}$ is an unbroken set, it also
follows that the right hand side of equation \eqref{bc4} is unbroken.


\begin{case}\label{re}
A Rotation Expansion move is performed.
\end{case}

In order for a Rotation Expansion move to be performed, we must have $a_{1} > b_{1}$.  Let $A$ be the first $b_{1}$ vertices of $M$, namely $A=\{v_{1}, \dots, v_{b_{1}}\}$, and let $B$ be what remains of the first $a_{1}$ vertices of $M$, namely $B=\{v_{b_{1} + 1}, \dots, v_{a_{1}}\}$.  After the Rotation Expansion is performed, let $A'$ be the first $a_{1} - b_{1}$ vertices of $M'$, namely $A'=\{v_1,\dots ,v_{a_1-b_1}\}$. Let $B'$ be what remains of the first $a_{1}$ vertices of $M'$, namely $B'=\{v_{a_1-b_1+1},\dots ,v_{a_1}\}$. Let $C'$ be the next $a_{1} - b_{1}$ vertices of $M'$, namely $C'=\{v_{a_{1} + 1}, \dots, v_{2a_{1} - b_{1}}\}$.

Let $k=a_1-1$ and for $i=1,\dots ,k$ let $\alpha_i=m(i+1,i)$, thus 
$\alpha_1,\dots ,\alpha_{k}$ are the simple eigenvalues for the
first top block of $M$. Let $p=b_1-1$, thus 
$-\alpha_1,\dots ,-\alpha_p$ are the simple eigenvalues for the
first bottom block of $M$.
Therefore, the simple eigenvalues for the first bottom block of $M'$ 
are $\alpha_k,\dots ,\alpha_1$.
The simple eigenvalues for
the first top block of $M'$ are 
$-\alpha_k,\dots ,-\alpha_1,\alpha_{p+1},\alpha_{p+2},\dots ,\alpha_k$.
The simple eigenvalues for every other block of $M'$ are the same 
as the simple eigenvalues of the block of corresponding size in $M$.
For example, see Figure 
\ref{fig:rotationexpansion} below where $k=4$ and $p=2$.

\begin{figure}[H]
\[\begin{tikzpicture}[scale=.9]
\tikzstyle{every node}=[draw,circle,fill=white,minimum size=4pt,inner sep=0pt]
\tikzset{->-/.style={decoration={markings,
mark=at position #1 with {\arrow[scale=1.5,>=stealth]{>}}},postaction={decorate}}}
\node[draw=none] at (1.45,-.19) {$-\alpha_1$};
\node[draw=none] at (2.45,-.19) {$-\alpha_2$};
\node[draw=none] at (1.5,.23) {$\alpha_1$};
\node[draw=none] at (2.5,.23) {$\alpha_2$};
\node[draw=none] at (3.5,.23) {$\alpha_3$};
\node[draw=none] at (4.5,.23) {$\alpha_4$};

\draw (1,0) node (1) {};
\draw (2,0) node (2) {};
\draw (3,0) node (3) {};
\draw (4,0) node (4) {};
\draw (5,0) node (5) {};
\draw (6,0) node (6) {};

\draw[dotted] (0.5,0) -- (6.5,0);
\draw [->-=.5] (5) to [bend right=70] (1);
\draw [->-=.5] (4) to [bend right=70] (2);
\draw [dashed] (6.5,.5) to [bend right=20] (6);
\draw [->-=.5] (1) to [bend right=70] (3);
\draw [dashed] (4) to [bend right=20] (4.5,-.5);
\draw [dashed] (5) to [bend right=20] (5.5,-.5);
\draw [dashed] (6) to [bend right=20] (6.5,-.5);

;\end{tikzpicture}
\hspace{1cm}
\begin{tikzpicture}[scale=.9]
\tikzstyle{every node}=[draw,circle,fill=white,minimum size=4pt,inner sep=0pt]
\tikzset{->-/.style={decoration={markings,
mark=at position #1 with {\arrow[scale=1.5,>=stealth]{>}}},postaction={decorate}}}
\node[draw=none] at (1.5,.2) {$-\alpha_4$};
\node[draw=none] at (2.5,.2) {$-\alpha_3$};
\node[draw=none] at (3.5,.2) {$-\alpha_2$};
\node[draw=none] at (4.5,.2) {$-\alpha_1$};
\node[draw=none] at (5.5,.2) {$\alpha_3$};
\node[draw=none] at (6.5,.2) {$\alpha_4$};
\node[draw=none] at (1.5,-.2) {$\alpha_4$};
\node[draw=none] at (2.5,-.2) {$\alpha_3$};
\node[draw=none] at (3.5,-.2) {$\alpha_2$};
\node[draw=none] at (4.5,-.2) {$\alpha_1$};

\draw (1,0) node (1) {};
\draw (2,0) node (2) {};
\draw (3,0) node (3) {};
\draw (4,0) node (4) {};
\draw (5,0) node (5) {};
\draw (6,0) node (6) {};
\draw (7,0) node (7) {};
\draw (8,0) node (8) {};

\draw[dotted] (0.5,0) -- (8.5,0);
\draw [->-=.5] (7) to [bend right=70] (1);
\draw [->-=.5] (6) to [bend right=70] (2);
\draw [->-=.5] (5) to [bend right=70] (3);
\draw [dashed] (8.5,.5) to [bend right=20] (8);
\draw [->-=.5] (1) to [bend right=70] (5);
\draw [->-=.5] (2) to [bend right=70] (4);
\draw [dashed] (6) to [bend right=20] (6.5,-.5);
\draw [dashed] (7) to [bend right=20] (7.5,-.5);
\draw [dashed] (8) to [bend right=20] (8.5,-.5);

;\end{tikzpicture}
\]

\caption{Rotation expansion applied to $M$ (left) to obtain $M'$ (right), with
simple eigenvalues shown}
\label{fig:rotationexpansion}
\end{figure}
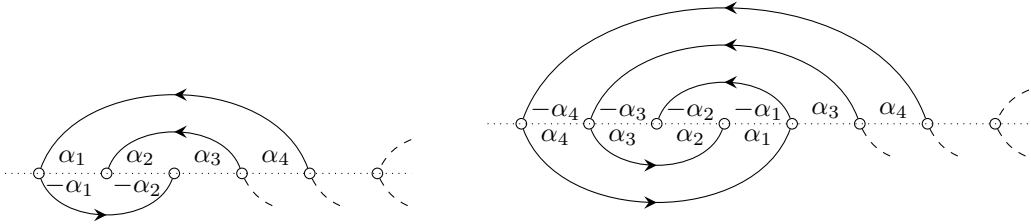


For the new bottom block of size $a_1$, clearly 
$\sright{A'\cup B'}=\sleft{A\cup B}$
and $\tright{A'\cup B'}=\tleft{A\cup B}$, so by induction both are unbroken.  For the new top block of size $2a_1-b_1$, we claim that:
\begin{eqnarray}
\label{re1} \sigma(\overleftarrow{A'}) &=& -\sigma(\overleftarrow{B}),\\
\label{re2} \sigma(\overleftarrow{B'}) &=& -\sigma(\overleftarrow{A}),\\
\label{re3} \sigma(A' \leftarrow B') &=& -\sigma(A \leftarrow B),\\
\label{re4} \sigma(\overleftarrow{C'}) &=& \sigma(\overleftarrow{B}),\\
\label{re5} \sigma(A' \leftarrow C') 
&=& [\sigma(\overleftarrow{B}) + 1] \cup [-\sigma(\overleftarrow{B}) + 1],\\
\label{re6} \sigma(B' \leftarrow C') &=& \sigma(A \leftarrow B) + 1,\\
\label{re7} \tleft{A'\cup B'\cup C'} &=& \left[\tleft{A\cup B}-1\right]
\cup S, \text{ where }S\subset\tleft{A\cup B}.
\end{eqnarray}

Equations \eqref{re1} through \eqref{re5} are obvious.
To prove equation \eqref{re6}, note that every element of $\sigma(B' \leftarrow C')$ has the form
\[-\alpha_{r}-\alpha_{r-1}- \dots -\alpha_{1}
+\alpha_{p+1}+\alpha_{p+2}+\dots +\alpha_{q},\]
for some integers $r$ and $q$ where $p\geq r\geq 0$ and $p+1\leq q\leq k$. 
Since $\sum_{i=1}^p-\alpha_i=1$ we have
\[(-\alpha_{r}-\alpha_{r-1}- \dots -\alpha_{1})
+\alpha_{p+1}+\alpha_{p+2}+\dots +\alpha_{q}
=(1+\alpha_{r+1}+\alpha_{r+2}+\dots +\alpha_p)
+\alpha_{p+1}+\alpha_{p+1}+\dots +\alpha_q.\]
As $r$ and $q$ vary, this generates the set $\sigma(A \leftarrow B) + 1$ as desired.

To prove equation \eqref{re7}, note that if $q$ is an integer such that $k+1\geq q\geq 1$, then
\[-\alpha_k-\alpha_{k-1}-\dots -\alpha_{q}
=-1+\alpha_1+\alpha_2+\dots +\alpha_{q-1}
\in\tleft{A\cup B}-1,\]
since $\sum_{i=1}^k-\alpha_i=-1$.
And if $r$ is an integer such that $p+1\leq r\leq k$, then
\[-\alpha_k-\alpha_{k-1}-\dots -\alpha_{1}
+\alpha_{p+1}+\alpha_{p+2}+\dots +\alpha_{r}
=-1+\alpha_{p+1}+\alpha_{p+2}+\dots +\alpha_{r}
=\alpha_1+\alpha_2+\dots +\alpha_r.\]
Letting $S=\{\alpha_1+\alpha_2+\dots +\alpha_r\mid p+1\leq r\leq k\}$,
we have $S\subset\tleft{A\cup B}$.

It follows from  
equations \eqref{re1}, \eqref{re2}, and \eqref{re3} that $-\sleft{A\cup B}$ 
is a subset of the eigenvalues in the new 
top block. By induction, this is an unbroken set that is symmetric about
minus one-half. By symmetry, the new top block must also include the eigenvalues
$\sleft{A\cup B}+1$. Therefore, the union of the right hand sides of
equations \eqref{re1} through \eqref{re6} is an unbroken set. Also,
since $\tleft{A\cup B}$ is unbroken by induction, the right hand side of
equation \eqref{re7} is also unbroken.


\begin{case}\label{pe}
A Pure Expansion move is performed.
\end{case}

Let $A$ be the first $a_1$ vertices of $M$, namely $A=\{v_1,\dots ,v_{a_1}\}$.
Let $B$ be the next $a_2$ vertices of $M$, 
namely $B=\{v_{a_1+1},\dots ,v_{a_1+a_2}\}$.
After the Pure Expansion is performed,
let $C'$ be the first $a_2$ vertices of $M'$,
namely $C'=\{v_1,\dots ,v_{a_2}\}$.
Let $A'$ be the next $a_1$ vertices of $M'$,
namely $A'=\{v_{a_2+1},\dots ,v_{a_2+a_1}\}$.
Let $B'$ be the next $a_2$ vertices of $M'$,
namely $B'=\{v_{a_2+a_1+1},\dots ,v_{a_1+2a_2}\}$.

Let $k=a_1-1$ and for $i=1,\dots ,k$ let $\alpha_i=m(i+1,i)$, thus 
$\alpha_1,\dots ,\alpha_{k}$ are the simple eigenvalues for the
first top block of $M$. Let $p=a_2-1$ and for $i=1,\dots ,p$ 
let $\epsilon_i=m(i+1,i)$, thus 
$\epsilon_1,\dots ,\epsilon_p$ are the simple eigenvalues for the
second bottom block of $M$. Let $\gamma=m(a_1,a_1+1)$. Note that
whenever $m(i,i+1)$ is a simple eigenvalue but $m(i+1,i)$ is not,
$m(i,i+1)$ must be positive (similarly $m(i+1,i)>0$ if $m(i+1,i)$
is a simple eigenvalue but $m(i,i+1)$ is not). Thus, by Lemma \ref{lem:simple},
$\gamma$ is either 1, 2, or 3.

The simple eigenvalues for the first bottom block of $M'$ 
are $\epsilon_p,\dots ,\epsilon_1$.
The simple eigenvalues for
the first top block of $M'$ are 
$-\epsilon_p,\dots ,-\epsilon_1,\gamma,\alpha_{1},\dots ,\alpha_k,-\gamma,
\epsilon_1,\dots ,\epsilon_p$.
The simple eigenvalues for every other block of $M'$ are the same 
as the simple eigenvalues of the block of corresponding size in $M$.
For example, see Figure 
\ref{fig:pureexpansion} below where $k=3$ and $p=2$.

\begin{figure}[H]
\[\begin{tikzpicture}[scale=.8]
\tikzstyle{every node}=[draw,circle,fill=white,minimum size=4pt,inner sep=0pt]
\tikzset{->-/.style={decoration={markings,
mark=at position #1 with {\arrow[scale=1.5,>=stealth]{>}}},postaction={decorate}}}
\node[draw=none] at (1.5,.17) {$\alpha_1$};
\node[draw=none] at (2.5,.17) {$\alpha_2$};
\node[draw=none] at (3.5,.17) {$\alpha_3$};
\node[draw=none] at (4.5,-.17) {$\gamma$};
\node[draw=none] at (5.5,.17) {$\epsilon_1$};
\node[draw=none] at (6.5,.17) {$\epsilon_2$};

\draw (1,0) node (1) {};
\draw (2,0) node (2) {};
\draw (3,0) node (3) {};
\draw (4,0) node (4) {};
\draw (5,0) node (5) {};
\draw (6,0) node (6) {};
\draw (7,0) node (7) {};
\draw (8,0) node (8) {};

\draw[dotted] (0.5,0) -- (8.5,0);
\draw [->-=.5] (4) to [bend right=70] (1);
\draw [->-=.5] (3) to [bend right=85] (2);
\draw [->-=.5] (7) to [bend right=70] (5);
\draw [dashed] (8.5,.5) to [bend right=20] (8);
\draw [dashed] (1) to [bend right=20] (1.5,-.5);
\draw [dashed] (2) to [bend right=20] (2.5,-.5);
\draw [dashed] (3) to [bend right=20] (3.5,-.5);
\draw [dashed] (4) to [bend right=20] (4.5,-.5);
\draw [dashed] (5) to [bend right=20] (5.5,-.5);
\draw [dashed] (6) to [bend right=20] (6.5,-.5);
\draw [dashed] (7) to [bend right=20] (7.5,-.5);
\draw [dashed] (8) to [bend right=20] (8.5,-.5);

;\end{tikzpicture}
\hspace{.5cm}
\begin{tikzpicture}[scale=.8]
\tikzstyle{every node}=[draw,circle,fill=white,minimum size=4pt,inner sep=0pt]
\tikzset{->-/.style={decoration={markings,
mark=at position #1 with {\arrow[scale=1.5,>=stealth]{>}}},postaction={decorate}}}
\node[draw=none] at (1.5,.18) {$-\epsilon_2$};
\node[draw=none] at (2.5,.18) {$-\epsilon_1$};
\node[draw=none] at (3.5,.18) {$\gamma$};
\node[draw=none] at (4.5,.18) {$\alpha_1$};
\node[draw=none] at (5.5,.18) {$\alpha_2$};
\node[draw=none] at (6.5,.18) {$\alpha_3$};
\node[draw=none] at (7.5,.18) {$-\gamma$};
\node[draw=none] at (8.5,.18) {$\epsilon_1$};
\node[draw=none] at (9.5,.18) {$\epsilon_2$};
\node[draw=none] at (1.5,-.18) {$\epsilon_2$};
\node[draw=none] at (2.5,-.18) {$\epsilon_1$};

\draw (1,0) node (1) {};
\draw (2,0) node (2) {};
\draw (3,0) node (3) {};
\draw (4,0) node (4) {};
\draw (5,0) node (5) {};
\draw (6,0) node (6) {};
\draw (7,0) node (7) {};
\draw (8,0) node (8) {};
\draw (9,0) node (9) {};
\draw (10,0) node (10) {};
\draw (11,0) node (11) {};

\draw[dotted] (0.5,0) -- (11.5,0);
\draw [->-=.5] (10) to [bend right=70] (1);
\draw [->-=.5] (9) to [bend right=70] (2);
\draw [->-=.5] (8) to [bend right=70] (3);
\draw [->-=.5] (7) to [bend right=70] (4);
\draw [->-=.5] (6) to [bend right=85] (5);
\draw [dashed] (11.5,.5) to [bend right=20] (11);
\draw [->-=.5] (1) to [bend right=70] (3);
\draw [dashed] (4) to [bend right=20] (4.5,-.5);
\draw [dashed] (5) to [bend right=20] (5.5,-.5);
\draw [dashed] (6) to [bend right=20] (6.5,-.5);
\draw [dashed] (7) to [bend right=20] (7.5,-.5);
\draw [dashed] (8) to [bend right=20] (8.5,-.5);
\draw [dashed] (9) to [bend right=20] (9.5,-.5);
\draw [dashed] (10) to [bend right=20] (10.5,-.5);
\draw [dashed] (11) to [bend right=20] (11.5,-.5);

;\end{tikzpicture}
\]

\caption{Pure expansion applied to $M$ (left) to obtain $M'$ (right), with
simple eigenvalues shown}
\label{fig:pureexpansion}
\end{figure}
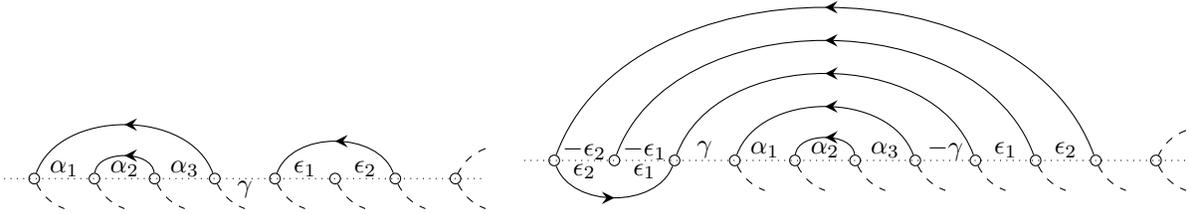

For the new bottom block of size $a_2$, clearly
$\sright{C'} = \sleft{B}$ and $\tright{C'}=\tleft{B}$,
so by induction both are unbroken.  For the new top block of size $a_1+2a_2$, we claim that:
\begin{eqnarray}
\label{pe1} \sleft{B'}&=&\sleft{B},\\
\label{pe2} \sleft{C'}&=&-\sleft{B},\\
\label{pe3} \sleft{A'}&=&\sleft{A},\\
\label{pe4} \sigma(C'\leftarrow B')&=&[\sigma(\overleftarrow{B}) + 1] \cup [-\sigma(\overleftarrow{B}) + 1],\\
\label{pe5} \sigma(C'\leftarrow A')&=&-\tleft{B}+\gamma+\tleft{A},\text
{ where }\gamma=1,2, \text{ or }3,\\
\label{pe6} \sigma(A'\leftarrow B')&=&-\sigma(C'\leftarrow A')+1,\\
\label{pe7} \tleft{C' \cup A' \cup B'}&=&\left[\tleft{B}-1\right]
\cup \left[\tleft{A}+\gamma-1\right]\cup\tleft{B}.
\end{eqnarray}

Equations \eqref{pe1} through \eqref{pe4} are obvious.
Equation \eqref{pe5} follows from the definitions of 
$\tleft{B}$ and $\tleft{A}$.
Since every eigenvalue in $\sigma(C'\leftarrow A')$ has the form
\[-\epsilon_r-\epsilon_{r-1}-\dots -\epsilon_{1}+\gamma
+\alpha_1+\alpha_2+\dots +\alpha_{q}\]
for some integers $r$ and $q$ such that $p\geq r\geq 0$ and $0\leq q\leq k$,
there is a unique eigenvalue in $\sigma(A'\leftarrow B')$ such that the 
sum of these two eigenvalues is 1, namely
\[\alpha_{q+1}+\alpha_{q+2}+\dots +\alpha_{k}-\gamma
+\epsilon_1+\epsilon_2+\dots +\epsilon_{r}.\]
Equation \eqref{pe6} follows from this.

Next, we show the union of the right hand sides of 
equations \eqref{pe1} through \eqref{pe6} is unbroken.
Since $\sleft{A}$ and $\sleft{B}$ contain 0 and are unbroken by induction,
it follows that the union of the right hand sides of 
equations \eqref{pe1} through \eqref{pe4} is unbroken
and contains both 0 and 1.
Since $\gamma=1,2,$ or $3$, $-\tleft{B}$ is unbroken and 
contains $-1$, and $\tleft{A}$ is unbroken and
contains $0$, we see that the right hand side of equation \eqref{pe5} 
is an unbroken set containing $0,1$, or $2$.
Thus the right hand side of equation \eqref{pe6} is unbroken and 
contains $1,0$, or $-1$, respectively.
It follows that the union of the right hand sides of 
equations \eqref{pe1} through \eqref{pe6} is unbroken.

To establish equation \eqref{pe7}, note that if $p+1\geq r\geq 1$, then
\[-\epsilon_p-\epsilon_{p-1}-\dots -\epsilon_{r}
=-1+\epsilon_1+\epsilon_2+\dots \epsilon_{r-1}
\in\tleft{B}-1,\]
because $\sum_{i=1}^p-\epsilon_i=-1$.
If $0\leq s\leq k$ then
\[-\epsilon_p-\epsilon_{p-1}-\dots -\epsilon_{1}+\gamma+
\alpha_1+\alpha_2+\dots +\alpha_{s}
=-1+\gamma+\alpha_1+\alpha_2+\dots +\alpha_{s}\in\tleft{A}+\gamma-1.\]
And if $0\leq q\leq p$, then
\[-\epsilon_p-\epsilon_{p-1}-\dots -\epsilon_{1}+\gamma+
\alpha_1+\alpha_2+\dots +\alpha_{k}-\gamma
+\epsilon_1+\epsilon_2+\dots +\epsilon_q
=\epsilon_1+\epsilon_2+\dots +\epsilon_q\in\tleft{B},\]
because $\sum_{i=1}^k\alpha_i=1$.  Now, since $\tleft{B}$ is unbroken and contains both 0 and 1, $\tleft{B}-1$ 
is unbroken and contains both $-1$ and 0, and $\tleft{A}+\gamma-1$ is unbroken
and contains 0, 1, or 2 (depending on the value of $\gamma$), we 
conclude that the union
of sets in the right hand side of equation \eqref{pe7} is unbroken.
\end{proof}

As noted in Example \ref{ex:principal}, it seems to be more generally true that the dimensions of
the eigenspaces of a Frobenius seaweed form a unimodal sequence. Furthermore, it appears that this unimodality property holds within each block.  More formally, we have the following conjecture.

\begin{conjecture}\label{conj:block}
Let $A$ be any top or bottom block in a Frobenius seaweed of $\mf{sl}(n)$ 
with $n\geq 2$, and if $d_i$ is the multiplicity of the 
eigenvalue $i$ in the multiset of eigenvalues contributed by $A$, then the sequence $\{d_i\}$ is unimodal about one half.
\end{conjecture}

\begin{remark}  The proof technique of Theorem \ref{GGConj} might 
be extended to establish Conjecture \ref{conj:block}.
However, this will likely require a refinement, or different set, of Winding-down moves.
Even so, extensive simulations suggest that the multiplicites do form a unimodal distribution.  In fact, more seems to be true.  Let $\lambda$ be a positive integer and let $F_\lambda$ be a Frobenius seaweed with spectrum $\{1-\lambda,\ldots ,\lambda\}$ and let 
$d_i(F_\lambda)$ equal the dimension of the eigenspace corresponding to the eigenvalue $i\in \{1-\lambda,\ldots ,\lambda\}$.  Consider now a sequence of such seaweeds $\{F_\lambda\}_{\lambda=1}^{\lambda=\infty}$ and an associated sequence of random variables  $\{X_\lambda\}_{\lambda=1}^{\lambda=\infty}$
defined by $P(X_\lambda = i)=\frac{d_i(F_\lambda)}{\dim F_\lambda}$.
We conjecture that the sequence of random variables $\{X_\lambda\}_{\lambda=1}^{\lambda=\infty}$ converges in distribution to a normal distribution
(see \textbf{\cite{Bender}}).
\end{remark}

\section{Afterword} Generally the eigenvalues of the adjoint of a principal element of a Frobenius Lie algebra $\mathfrak{g}$ need not be integers (see 
\textbf{\cite{Diatta}} for examples).  Here, we have shown that if $\mathfrak{g}$ is a Frobenius seaweed subalgebra of $\mathfrak{sl}(n)$ then its spectrum must consist of an unbroken multiset of integers with multiplicities forming a symmetric distribution about one half.

In a forthcoming sequence of articles, we show that Theorem \ref{GGConj} is also true if one considers Type-B, Type-C, and Type-D Frobenius seaweeds.  The argument in the Type-B and Type-C cases is made interesting (in different ways) by the existence of an exceptional root, while the Type-D case is complicated by the bifurcation point in its associated Dynkin diagram.

Finally, we note that while the unbroken spectrum is a property of Frobenius seaweed subalgebras of $\mathfrak{sl}(n)$ it is not characteristic.  Many Frobenius Lie poset subalgebras \textbf{\cite{CG}} of $\mathfrak{sl}(n)$ have the unbroken property but are not seaweed subalgebras.  See the following example. 
 
\begin{example}  Consider the poset $\cal P =$ $\{1,2,3,4\}$ with $1,2 \preceq 3 \preceq 4$ and no relations other than those following from these.  Letting 
$\mathbb{C}$ be the ground field, one may construct an 
associative matrix algebra $A(\cal P, \mathbb{C})$ 
which is the span over $\mathbb{C}$ of $e_{ij}$, $i\preceq j$ with multiplication given by
$e_{i,j}e_{l,k}=e_{i,k}$ if $j=l$ and 0 otherwise.  The associative algebra $A(\cal P, \mathbb{C})$ becomes a Lie algebra $\mathfrak{g}(\cal P, \mathbb{C})$ under commutator multiplication and, if one considers only the elements of trace 0, may be regarded as a Lie subalgebra of 
$\mathfrak{sl}(4)$:  In fact, a Frobenius Lie subalgebra with Frobenius functional $F= e_{1,4}^* + e_{2,4}^*+e_{2,3}^*$, principal element 
$\hat{F}=\rm{diag}\left(\frac{1}{2},\frac{1}{2},-\frac{1}{2},-\frac{1}{2}\right)$, and spectrum $\{0,0,0,0,1,1,1,1\}$. The algebra $\mathfrak{g}(\cal P, \mathbb{C})$ has rank 3 and dimension 8. However, the only Frobenius seaweed subalgebra of $\mathfrak{sl}(4)$ with the same rank and dimension is of type $\frac{2|2}{1|3}$, but the latter has spectrum given by  the multiset $\{-1,0,0,0,1,1,1,2\}$.
\end{example}

\subsection*{Acknowledgment}

The authors would like to thank Murray Gerstenhaber for many insightful conversations.  We also thank Jim Stasheff for giving the first author multiple occasions to speak on these topics at the University of Pennsylvania's Deformation Theory Seminar.  We are also indebted to Anthony Giaquinto and Aaron Lauve for pointing out an error in a previous version of this work.

\end{document}